\documentclass[12pt,a4paper]{amsart}
\usepackage{amsopn,amssymb,mathrsfs}
\newtheorem{thm}{Theorem}[section]

\newtheorem{lem}[thm]{Lemma}
\newtheorem{prop}[thm]{Proposition}

\theoremstyle{definition}
\newtheorem{defn}[thm]{Definition}
\newtheorem{example}[thm]{Example}
\newtheorem{rem}[thm]{Remark}

\newcommand{\ball}[1]{\ensuremath{B_{#1}}}

\newcommand{\Ck}[1]{\ensuremath{\mathscr{C}({#1})}}
\newcommand{\closure}[1]{\ensuremath{\overline{{#1}}}}

\newcommand{\czerok}[1]{\ensuremath{c_0({#1})}}

\newcommand{\dind}[1]{\ensuremath{\mbox{\boldmath{$1$}}_{#1}}}
\newcommand{\dual}[1]{\ensuremath{{#1}^*}}

\newcommand{\ind}[1]{\ensuremath{\chi_{#1}}}

\newcommand{\lp}[1]{\ensuremath{\ell_{#1}}}

\newcommand{\lpk}[2]{\ensuremath{\ell_{#1}({#2})}}

\newcommand{\mapping}[3]{\ensuremath{{#1}:{#2}\longrightarrow{#3}}}

\newcommand{\nat}{\mathbb{N}}
\newcommand{\norm}[1]{\ensuremath{\left|\left|{#1}\right|\right|}}
\newcommand{\normdot}{\ensuremath{\left|\left|\cdot\right|\right|}}

\newcommand{\oneton}[2]{\ensuremath{{#1}_1,\ldots,{#1}_{#2}}}

\newcommand{\pnorm}[2]{\ensuremath{\left|\left|{#1}\right|\right|_{#2}}}
\newcommand{\pnormdot}[1]{\ensuremath{\left|\left|\cdot\right|\right|_{#1}}}

\newcommand{\rat}{\mathbb{Q}}
\newcommand{\real}{\mathbb{R}}
\newcommand{\res}[2]{\ensuremath{\res({#1};{#2})}}
\newcommand{\restrict}[1]{\ensuremath{\!\!\upharpoonright_{#1}}}
\newcommand{\setcomp}[2]{\ensuremath{\left\{{#1}\;:\;\,{#2}\right\}}}

\newcommand{\trinorm}[1]{\ensuremath{\left|\left|\left|{#1}\right|\right|\right|}}
\newcommand{\trinormdot}{\ensuremath{\left|\left|\left|\cdot\right|\right|\right|}}

\newcommand{\weakstar}{\ensuremath{w^*}}
\newcommand{\wone}{\ensuremath{\omega_1}}

\DeclareMathOperator{\aspan}{span} %
\DeclareMathOperator{\card}{card} %

\begin{document}

\title{Unconditional bases and strictly convex dual renormings}
\author{R.\ J.\ Smith}
\address{Institute of Mathematics of the AS CR, \v{Z}itn\'{a} 25, CZ - 115 67 Praha 1, Czech Republic}
\email{smith@math.cas.cz}
\author{S.\ Troyanski}
\address{Departamento de Matemáticas, Universidad de Murcia, Campus de Espinardo.\\30100 Murcia, Spain}
\thanks{Supported by MTM 2008-05396/MTM/Fondos FEDER and Fundati\'on S\'eneca
0690/PI/04 CARM and Institute of Mathematics and Informatics
Bulgarian Academy of Sciences, and grant MM-1401/2004 of Bulgarian
NSF}
\email{stroya@um.es}

\subjclass[2000]{46B03; 46B26; 46B15}
\date{November 2008}
\keywords{Strictly convex, norm, unconditional basis}

\begin{abstract}
We present equivalent conditions for a space $X$ with an
unconditional basis to admit an equivalent norm with a strictly
convex dual norm.
\end{abstract}

\maketitle

\section{Introduction}

We say that a norm $\normdot$ on a Banach space $X$ is {\em strictly
convex}, or {\em rotund}, if $x = y$ whenever $\norm{x} = \norm{y} =
\frac{1}{2}\norm{x + y}$. Geometrically, $\normdot$ is strictly
convex if its unit sphere contains no non-trivial line segments. The
norms of many classical spaces fail to possess this property,
however, it is often possible to introduce a new, equivalent norm
that does. Therefore we seek verifiable conditions that allow us to
determine when such a {\em renorming} is possible.

While the notion of strict convexity has been established now for
several decades, it has eluded general characterisation. In this
note, we study the class of spaces $X$ with an unconditional basis
(generally uncountable). In section \ref{dualr}, we obtain
equivalent conditions for $X$ to admit an equivalent norm, such that
its dual norm is strictly convex. The tools used are topological.

It should be noted that, in the context of spaces with unconditional
bases, equivalent conditions for the existence of some other types
of norm are known. We say that $\normdot$ is {\em locally uniformly
rotund}, or {\em LUR}, if, given $x$ and $x_n$ in $X$ such that
$\norm{x} = \norm{x_n} = 1$ and $\norm{x + x_n} \rightarrow 2$, we
have $\norm{x - x_n} \rightarrow 0$. Clearly, if $\normdot$ is LUR
then it is also strictly convex. The norm $\normdot$ is said to be
{\em G\^{a}teaux smooth} if, given non-zero $x$, we have
$$
\lim_{t \rightarrow 0} \frac{\norm{x + th} + \norm{x - th} -
2\norm{x}}{t} \;=\; 0.
$$
If, for all non-zero $x$, this limit exists uniformly for $h$ in the
unit sphere of $X$, then $\normdot$ is {\em Fr\'{e}chet smooth}. By
a well known result of \v{S}mulyan (cf \cite[Theorem
I.1.4]{dgz:93}), if the dual norm of $\normdot$ is strictly convex
(respectively LUR) then $\normdot$ is G\^{a}teaux (respectively
Fr\'{e}chet) smooth. The converses do not hold, even up to
renormings. In fact, there exists a space with a Fr\'{e}chet smooth
norm, which does not admit any equivalent norm with a strictly
convex dual norm (cf \cite[Theorem VII.5.2 (ii)]{dgz:93}). However,
in the class of spaces with unconditional bases, we do have
equivalence up to a renorming.

\begin{thm}\label{flurequiv}
Let $X$ have an unconditional basis. Then $X$ admits an equivalent
norm with LUR dual norm whenever $X$ admits an equivalent
Fr\'{e}chet smooth norm.
\end{thm}

This result has been known since the 1960s. Indeed, in
\cite{troy:68}, the second named author proved that if $X$ has a
unconditional basis then $X$ admits an equivalent LUR norm. Since
\cite{troy:68} is written in Russian and is of limited availability,
for convenience, we define this LUR norm here. Let $\normdot$ be the
original norm on $X$ and $(e_\gamma)_{\gamma \in \Gamma}$ an
unconditional basis with conjugate system $(f_\gamma)_{\gamma \in
\Gamma}$. By a renorming, we may assume that $(e_\gamma)_{\gamma \in
\Gamma}$ is 1-unconditional with respect to $\normdot$. Recall Day's
norm $\pnormdot{\mathrm{Day}}$ on $\lpk{\infty}{\Gamma}$ (cf
\cite[Definition II.7.2]{dgz:93}) and the fact that
$\pnormdot{\mathrm{Day}}$ is LUR when restricted to
$\czerok{\Gamma}$ (cf \cite[Theorem II.7.3]{dgz:93}). Define
$\mapping{T}{X}{\czerok{\Gamma}}$ by $(Tx)(\gamma) = f_\gamma(x)$
for $x \in X$ and $\gamma \in \Gamma$, and set
$$
\trinorm{x}^2 \;=\; \pnorm{Tx}{\mathrm{Day}}^2 + \sum_{n=0}^\infty
2^{-n}\pnorm{x}{n}^2
$$
where
$$
\pnorm{x}{n}^2 \;=\; \sup\setcomp{\norm{\sum_{\gamma \in
\Gamma\setminus A} f_\gamma(x)e_\gamma} + 2\sum_{\gamma \in A}
|f_\gamma(x)|}{A \subseteq \Gamma \mbox{ and }\card{A} \leq n}.
$$
Then $\trinormdot$ is LUR.

\begin{proof}[Proof of Theorem \ref{flurequiv}] Let $X$ have a
1-unconditional basis $(e_\gamma)_{\gamma \in \Gamma}$, with
conjugate system $(f_\gamma)_{\gamma \in \Gamma}$. If $X$ admits an
equivalent Fr\'{e}chet smooth norm then it cannot contain an
isomorphic copy of $\lp{1}$, for $\lp{1}$ admits no such norm (cf
\cite[Corollary II.3.3]{dgz:93}). Hence $(e_\gamma)_{\gamma \in
\Gamma}$ is shrinking (cf \cite[Theorem 1.c.9]{jl:77}), and it
follows that the conjugate system $(f_\gamma)_{\gamma \in \Gamma}$
is an unconditional basis of $\dual{X}$. Now it is a straightforward
matter to verify that the LUR norm $\trinormdot$ defined above, but
on $\dual{X}$, is $\weakstar$-lower semicontinuous and thus the dual
of an equivalent norm on $X$.
\end{proof}

In particular, if $X$ has an unconditional basis then $X$ admits an
equivalent norm with a LUR dual norm if and only if $X$ does not
cannot contain any isomorphic copies of $\lp{1}$. In the dual
strictly convex case, $X$ may contain copies of $\lp{1}$ but, very
roughly speaking, it cannot contain too many of them. We make this
statement clearer in Remark \ref{notmuchl1}. Note also that if
$\Gamma$ is uncountable and $X$ contains an isomorphic copy of
$\lpk{1}{\Gamma}$, then it is known that $X$ cannot admit any
equivalent norm with a strictly convex dual norm. Indeed,
$\lpk{1}{\Gamma}$ does not even admit an equivalent G\^{a}teaux
smooth norm \cite[Theorem 9]{day:55}.

In Section \ref{dualr}, we present our main topological tools, Lemma
\ref{master} and Theorem \ref{adequate}, together with Theorem
\ref{fullcharac}, our characterisation. In section \ref{examples},
we apply the topological tools to examples from the literature.

\section{Strictly convex dual norms}\label{dualr}

Let $\Gamma$ be a set, and suppose that the real Banach space $X$
has a normalised unconditional basis $(e_\gamma)_{\gamma \in
\Gamma}$. We shall identify both $X$ and $\dual{X}$ as vector
sublattices of $\real^\Gamma$ in the natural way, with the pointwise
lattice order. Moreover, both sublattices are {\em ideals}, in the
sense that if $x \in \real^\Gamma$, $y \in X$ and $|x| \leq y$, then
$x \in X$, and similarly for $\dual{X}$. By a renorming, we can
assume that the basis is 1-unconditional, which means that $X$ and
$\dual{X}$ are both Banach lattices. It will be helpful to keep the
lattice structure of $X$ and $\dual{X}$ in mind. It is clear that
the dual norm of $\dual{X}$ is always finer than
$\pnormdot{\infty}$, which we define on $\dual{X}$ in the obvious
way.

We must define some topological concepts. A function $\mapping{d}{Z
\times Z}{[0,\infty)}$ is called a {\em symmetric} if it satisfies
all the axioms for a metric, with the possible exception of the
triangle inequality. Let $(Z,\tau)$ be a regular topological space.
We say that $Z$ is {\em fragmentable} if there exists a metric $d$
on $Z$ with the property that, for any non-empty subset $E \subseteq
Z$ and $\varepsilon > 0$, there exists a $\tau$-open set $U$ such
that $E \cap U$ is non-empty and has $d$-diameter not exceeding
$\varepsilon$. In fact, by \cite[Remark 1.10]{rib:87}, for $Z$ to be
fragmentable, we only require that $d$ is a non-negative function on
$Z \times Z$, with the property that $x = y$ whenever $d(x,y) = 0$.
A family of subsets $\mathscr{N}$ of $Z$ is a {\em network} for $Z$
if, given $x \in U \in \tau$, there exists $N \in \mathscr{N}$ such
that $x \in N \subseteq U$. A family of subsets $\mathscr{F}$ is
called {\em isolated} if $E \cap
\closure{\bigcup\mathscr{F}\backslash\{E\}}$ is empty whenever $E
\in \mathscr{F}$; equivalently, there is some $U \in \tau$ such that
$E \subseteq U$ and $U \cap F$ is empty for all $F \in
\mathscr{F}\backslash\{E\}$. A network $\mathscr{N}$ is called
$\sigma${\em-isolated} if it can be written as $\mathscr{N} =
\bigcup_{n = 1}^\infty \mathscr{N}_n$, where each $\mathscr{N}_n$ is
isolated. We will say that $Z$ is {\em descriptive} if it is compact
and admits a $\sigma$-isolated network. The class of descriptive
compact spaces is broad and encompasses all metrisable compacta, as
well as {\em Eberlein} and {\em Gul'ko} compacta, which we shall
consider later, in Section \ref{examples}. Symmetrics,
$\sigma$-isolated networks and fragmentable and descriptive spaces
have been studied in, for example,
\cite{gru:84,hansell:01,raja:03,rib:87}. We say that $Z$ is a {\em
Gruenhage space} if there exist families $(\mathscr{U}_n)_{n \in
\nat}$ of open sets with the property that given distinct $x,y \in
Z$, there exist $n \in \nat$ and $U \in \mathscr{U}_n$ such that $U
\cap \{x,y\}$ is a singleton, and either $x$ lies in finitely many
$U^\prime \in \mathscr{U}_n$, or $y$ lies in finitely many $U^\prime
\in \mathscr{U}_n$. Gruenhage spaces were introduced in
\cite{gru:87} and have recently found application in the theory of
strictly convex dual norms \cite{smith:08}. Every descriptive
compact space is Gruenhage. We let $\tau_p$ denote the pointwise
topology on $\real^\Gamma$. We will introduce further classes of
compact topological spaces in Section \ref{examples}.

\begin{lem}
\label{master} Let $K \subseteq [0,1]^{\Gamma}$ be $\tau_p$-compact
satisfy $x \wedge y \in K$ whenever $x,y \in K$. Suppose further
that there exists a $\tau_p$-lower semicontinuous function
$\mapping{\rho}{K}{[0,1]}$ such that
\begin{enumerate}
\item[($*$)] \label{rhocondition} if $y < x$ then there exists $\alpha <
\rho(x)$ and an open set $U \ni y$ with the property that if $z \leq
x$ and $z \in U$, then $\rho(z) < \alpha$.
\end{enumerate}
Then we deduce the following
\begin{enumerate}
\item[I.] $K$ is fragmentable;
\item[II.] for all $r \in [0,1]$, $(\rho^{-1}(r),\tau_p)$ has a
$\sigma$-isolated network.
\end{enumerate}
Moreover, if $K \subseteq \{0,1\}^{\Gamma}$, then
\begin{enumerate}
\item[III.] $K$ is descriptive.
\end{enumerate}
\end{lem}

\begin{proof}
(I). For $x,y \in K$, set $d(x,y) = \max\{\rho(x),\rho(y)\} - \rho(x
\wedge y)$. Note that ($*$) implies that $\rho$ is strictly
increasing, that is, $\rho(y) < \rho(x)$ whenever $y < x$. Thus, if
$d(x,y) = 0$ then $x = x \wedge y = y$, so $d$ is a symmetric. We
show that it fragments $K$. Indeed, if $M \subseteq K$ is non-empty
and $\varepsilon
> 0$, let $\alpha = \sup_{x \in M}\rho(x)$. Take $x \in M$ such that
$\rho(x) > \alpha - \varepsilon$. By the lower semicontinuity, there
exists an open set $U \ni x$ such that $\rho(y) > \rho(x) -
\varepsilon$ whenever $y \in U$. Moreover, we can assume that $y
\wedge z \in U$ whenever $y,z \in U$. In particular, if $y,z \in M
\cap U$ then $d(y,z) \leq \alpha - \rho(y \wedge z) < \alpha -
(\alpha - 2\varepsilon) = 2\varepsilon$. By \cite[Remark
1.10]{rib:87} mentioned in the preamble to this section, $K$ is
fragmentable.

(II). For $x \in K$ and $\varepsilon > 0$, define
$$
B(x,\varepsilon) \;=\; \setcomp{y \in K}{\rho(y) \leq \rho(x)\mbox{
and }d(x,y) < \varepsilon}.
$$
Since $y \mapsto \rho(x \wedge y)$ is $\tau_p$-lower semicontinuous,
$B(x,\varepsilon)$ is open in $\rho^{-1}[0,\rho(x)]$. We prove that
if the sequence $(x_n)$ satisfies $\max\{\rho(x),\rho(x_n)\}
\rightarrow \rho(x)$ and $d(x,x_n) \rightarrow 0$, then $x_n
\rightarrow x$. Indeed, first suppose that $(x_{n_r})$ is a
subsequence converging to some $y \in K$. We have $\rho(y) \leq
\lim\inf \rho(x_{n_r}) \leq \rho(x)$. We claim that $x \leq y$. For
a contradiction, suppose otherwise. Since $x_{n_r} \rightarrow y$,
we have $x \wedge x_{n_r} \rightarrow x \wedge y < x$. Thus by
($*$), there exists $\alpha < \rho(x)$ and an open set $U \ni x
\wedge y$ such that if $z \leq x$ and $z \in U$ then $\rho(z) <
\alpha$. It follows that
\begin{eqnarray*}
d(x,x_{n_r}) &=& \max\{\rho(x),\rho(x_{n_r})\} - \rho(x \wedge
x_{n_r})\\
&\geq& \rho(x) - \rho(x \wedge x_{n_r})\\
&>& \rho(x) - \alpha\\
&>& 0
\end{eqnarray*}
for large enough $r$, which is a contradiction. This proves our
claim that $x \leq y$, and since $\rho(y) \leq \rho(x)$ and $\rho$
is strictly increasing, we conclude that $y = x$. Being fragmentable
and compact, $K$ is also sequentially compact, so $x_n \rightarrow
x$. Thus, if $x \in U \subseteq \rho^{-1}(r)$, with $U$ open in
$\rho^{-1}(r)$, then there exists $\varepsilon > 0$ such that
$B(x,\varepsilon) \cap \rho^{-1}(r) \subseteq U$. Consequently, $d$
semi-metrises $\rho^{-1}(r)$ for every $r$. Since $K$ is compact, it
is fragmentable by a metric which generates a finer topology than
$\tau_p$ \cite[Corollary 1.11]{rib:87}, and thus each $\rho^{-1}(r)$
is so fragmented. It follows that $\rho^{-1}(r)$ has a
$\sigma$-isolated network by \cite[Theorems 9.8 and 5.11]{gru:84}
and \cite[Lemma 2.2]{raja:03}.

(III). Suppose now that $K \subseteq \{0,1\}^\Gamma$. We can and do
assume that $\rho(K\backslash\{\ind{\varnothing}\}) \subseteq
[\frac{1}{2},1]$. If we define
$$
M = \setcomp{\lambda\ind{A}}{\ind{A} \in K, \lambda \in [0,2]},
$$
$\sigma(\lambda \ind{A}) = \lambda \rho(A)$ for $A \neq
\varnothing$, $\lambda > 0$, and $\sigma(\ind{\varnothing}) = 0$, it
is straightforward to verify that $M$ is $\tau_p$-compact and
$\sigma$ is a $\tau_p$-lower semicontinuous function on $M$. It is
clear that $x \wedge y \in M$ whenever $x,y \in M$. To see that
$\sigma$ satisfies ($*$) too, we take $x = \lambda \ind{A}, y =
\mu\ind{B} \in M$ such that $y < x$. Clearly $A$ is non-empty and
$\lambda > 0$. In addition, we can assume that either (a) $A = B$
and $0 < \mu < \lambda$, or (b) $B \subsetneqq A$ and $\mu \leq
\lambda$. If (a) holds then put $\beta = \frac{1}{2}(\lambda +
\mu)\rho(\ind{A}) < \sigma(x)$, select $\gamma \in A$ and define the
open set $V = \setcomp{z \in M}{0 < z_\gamma < \frac{1}{2}(\lambda +
\mu)}$. If $z = \nu\ind{C} \leq x$ and $z \in V$ then $C \subseteq
A$ and $\nu < \frac{1}{2}(\lambda + \mu)$, giving $\sigma(z) <
\frac{1}{2}(\lambda + \mu)\rho(\ind{A}) = \beta$. If (b) holds then
take $\alpha < \rho(\ind{A})$ and an open neighbourhood $U \subseteq
K$ of $\ind{B}$ with the property that $\rho(\ind{C}) < \alpha$
whenever $\ind{C} \in U$. Let $\beta = \frac{\lambda}{2}(\alpha +
\rho(\ind{A})) < \sigma(x)$, and observe that $V =
\setcomp{\nu\ind{C}}{\ind{C} \in U, \nu < \beta/\alpha}$ is open in
$M$. Moreover, $y \in V$ and, if $z = \nu\ind{C} \leq x$ and $z \in
V$ then $\sigma(z) < \beta$.

By (II), $(S,\tau_p)$ has a $\sigma$-isolated network, where $S =
\sigma^{-1}(1)$. Following \cite[Theorem 7.2]{hansell:01},
$(L,\tau_p)$ also has a $\sigma$-isolated network, where $L =
\sigma^{-1}[0,1]$. We sketch an argument for completeness. The map
$(t,x) \mapsto (t,tx)$ is a homeomorphism of $(0,1] \times S$ and a
subset $E$ of $(0,1] \times L\backslash\{0\}$, so $E$ has a
$\sigma$-isolated network. If $y = \lambda \ind{A} \in
L\backslash\{0\}$ then $x = \frac{1}{\sigma(y)} y =
\frac{1}{\rho(\ind{A})}\ind{A}$ is an element of $M$ because
$\rho(\ind{A}) \geq \frac{1}{2}$. Moreover $x \in S$ and
$(\sigma(y),y) \in E$. Thus $E$ projects onto $L\backslash\{0\}$,
and so $L\backslash\{0\}$ has a $\sigma$-isolated network, again by
the proof of \cite[Theorem 7.2]{hansell:01}. To finish, we simply
adjoin $\{0\}$ to the network. Since $K$ embeds in $L$, we have
proved (III).
\end{proof}

Lemma \ref{master} will be applied first to a specific class of
topological spaces and an associated class of Banach spaces.

\begin{defn}\label{adequatedef}
A family of subsets $\mathscr{A}$ of $\Gamma$ is called {\em
adequate on} $\Gamma$, or simply {\em adequate}, if it satisfies the
following three conditions:
\begin{enumerate}
\item $\{\gamma\} \in \mathscr{A}$ for all $\gamma \in \Gamma$;
\item $B \in \mathscr{A}$ whenever $A \in \mathscr{A}$ and $B \subseteq
A$;
\item\label{compact} if every finite subset of $A \subseteq \Gamma$ is in $\mathscr{A}$
then $A \in \mathscr{A}$.
\end{enumerate}
\end{defn}

Note that we can replace property (\ref{compact}) of Definition
\ref{adequatedef} with
\begin{enumerate}
\item[(3$^\prime$)] $K_{\mathscr{A}}$ is a $\tau_p$-compact subset of $\{0,1\}^\Gamma$, where $K_\mathscr{A} = \setcomp{\ind{A}}{A \in
\mathscr{A}}$.
\end{enumerate}

The set of totally ordered subsets of a partially ordered set $E$ is
adequate on $E$. If $\Gamma$ is a set of consistent sentences in a
first-order theory then the family of consistent subsets of $\Gamma$
is adequate on $\Gamma$. Adequate families were defined in
\cite{tal:79} and have been considered subsequently by several
authors in, for example, \cite{am:93,ls:84}. Given an adequate
family $\mathscr{A}$, we define Banach lattice ideal
$\lp{\mathscr{A}}$ as the set of all $x \in \lpk{\infty}{\Gamma}$
satisfying $\pnorm{x}{\mathscr{A}} < \infty$, where
$\pnorm{x}{\mathscr{A}} = \sup_{A \in \mathscr{A}}
\pnorm{x\restrict{A}}{1}$, where $\pnormdot{1}$ is the standard
1-norm (cf \cite[Definition 2.1]{am:93}). For example, if
$\mathscr{A} = \{\varnothing\} \cup \setcomp{\{\gamma\}}{\gamma \in
\Gamma}$ then $\lp{\mathscr{A}} = \lpk{\infty}{\Gamma}$, and if
$\Gamma \in \mathscr{A}$ then $\lp{\mathscr{A}} = \lpk{1}{\Gamma}$.
It is easy to see that, in general, the standard unit vectors
$(e_\gamma)_{\gamma \in \Gamma}$ form a normalised 1-unconditional
basic sequence in $\lp{\mathscr{A}}$. We set $h_\mathscr{A} =
\closure{\aspan}^{\pnormdot{\mathscr{A}}} (e_\gamma)_{\gamma \in
\Gamma}$ and denote the dual norm on $h_\mathscr{A}^*$ again by
$\pnormdot{\mathscr{A}}$. Given $x \in h_\mathscr{A}$ and $A
\subseteq \Gamma$, we define
$$
\dind{A}(x) \;=\; \sum_{\gamma \in A} x_\gamma
$$
whenever this sum makes sense. It is clear that the functions
$\dind{A}$, $A \in \mathscr{A}$, lie in $h_\mathscr{A}^*$, with
$\pnorm{\dind{A}}{\mathscr{A}} = 1$ whenever $A$ is non-empty. It is
also easy to verify that the map $\pi:\ind{A} \mapsto \dind{A}$ on
$K_{\mathscr{A}}$ is $\tau_p$-$\weakstar$ continuous; in particular,
the image $\pi(K_\mathscr{A})$ is homeomorphic to $K_{\mathscr{A}}$.
The proof of Theorem \ref{adequate} below requires some renorming
results, which we state partially.

\begin{thm}[{\cite{smith:08}}] \label{gru} $\;$
\begin{enumerate}
\item Let $K$ be a Gruenhage compact space. Then $\Ck{K}$ admits a
norm with a strictly convex dual norm.
\item Let $(\ball{\dual{X}},\weakstar)$ be a Gruenhage compact space.
Then $X$ admits a norm with a strictly convex dual norm.
\end{enumerate}
\end{thm}

\begin{thm}[{\cite[Theorem 2.6]{motz:06}}]
\label{sclattice} Let $(X,\normdot)$ be a Banach lattice ideal of
$\real^\Gamma$, such that $\pnormdot{\infty} \leq \normdot$. Then
$X$ admits a $\tau_p$-lower semicontinuous, strictly convex norm if
and only if $X$ admits a $\tau_p$-lower semicontinuous norm
$\trinormdot$ satisfying $\trinorm{x} < \trinorm{y}$ whenever $|x| <
|y|$.
\end{thm}

\begin{thm}
\label{adequate} Let $\mathscr{A}$ be an adequate family. Then the
following are equivalent.
\begin{enumerate}
\item $K_{\mathscr{A}}$ is a Gruenhage compact;
\item $h_\mathscr{A}$ admits a norm with strictly convex dual norm;
\item there exists a strictly increasing, $\tau_p$-lower semicontinuous
map $\mapping{\rho}{K_{\mathscr{A}}}{[0,1]}$;
\item $K_{\mathscr{A}}$ is a descriptive compact.
\end{enumerate}
\end{thm}

\begin{proof}
(1) $\Rightarrow$ (2). If $\mathscr{A}$ is a Gruenhage compact then
$\Ck{\mathscr{A}}$ admits an equivalent norm $\normdot$ with a
strictly convex dual norm by Theorem \ref{gru}, part (1). Define
$\mapping{T}{h_\mathscr{A}}{\Ck{\mathscr{A}}}$ by $(Tx)(A) =
\ind{A}(x)$ and observe that, since $\mathscr{A}$ is adequate, we
have $\frac{1}{2}\pnorm{x}{\mathscr{A}} \leq \pnorm{Tx}{\infty} \leq
\pnorm{x}{\mathscr{A}}$. Consequently, the dual norm of
$\trinormdot$, where $\trinorm{x} = \norm{Tx}$, $x \in
h_\mathscr{A}$, is strictly convex. (2) $\Rightarrow$ (3) follows
from Theorem \ref{sclattice} and the natural embedding of
$K_\mathscr{A}$ in $\ball{h^*_\mathscr{A}}$ defined above. (3)
$\Rightarrow$ (4). If $\ind{B} < \ind{A}$ then take $\gamma \in
A\backslash B$. Property ($*$) of Lemma \ref{master} is fulfilled by
setting $\alpha = \rho(A\backslash\{\gamma\})$ and $U =
\setcomp{\ind{C}}{\gamma \notin C}$. (4) $\Rightarrow$ (1) follows
from \cite[Corollary 4]{smith:08}.
\end{proof}

\begin{rem}
The proof of the implication (1) $\Rightarrow$ (4) in Theorem
\ref{adequate} uses a renorming of a Banach space. The authors would
be interested to see a direct, purely topological proof of this
result. The proof of Theorem \ref{adequate} also shows that if
$\mathscr{A}$ is adequate then $\Ck{K_{\mathscr{A}}}$ admits a norm
with strictly convex dual norm if and only if $K_{\mathscr{A}}$ is a
Gruenhage space. We don't know if the direct implication holds in
general.
\end{rem}

We finish this section by providing a characterisation of spaces
with unconditional bases which admit an equivalent norm with
strictly convex dual norm. Let $X$ have a normalised,
1-unconditional basis $(e_\gamma)_{\gamma \in \Gamma}$. Let
$\normdot$ denote the dual norm on $\dual{X}$. Define
$$
\mathscr{A} \;=\; \setcomp{A \subseteq \Gamma}{\dind{A} \in
\dual{X}}.
$$
The family $\mathscr{A}$ contains all singletons $\{\gamma\}$,
$\gamma \in \Gamma$, and is closed under taking subsets and finite
unions. Hence $\mathscr{A}$ is adequate if and only if $X$ is
isomorphic to $\lpk{1}{\Gamma}$. While $K_\mathscr{A} =
\setcomp{\ind{A}}{A \in \mathscr{A}}$ is not compact in general, it
is $\sigma$-compact because $K_A  = \bigcup_{n=1}^\infty
K_{\mathscr{A}_n}$, where $\mathscr{A}_n$ is the adequate family
$$
\setcomp{A \in \mathscr{A}}{\norm{\dind{A}} \leq n}.
$$
Let $W$ be the linear span of
$$
\setcomp{\dind{A}}{A \in \mathscr{A}, n \in \nat}.
$$
While it is not true that $\closure{W}^{\normdot} = \dual{X}$ in
general, it is clear that $\closure{W}^{\pnormdot{\infty}} =
\dual{X}$. We require the following result.

\begin{prop}[{\cite[Corollary 10]{smith:08}}]
\label{coarserdense} Let $X$ be a Banach space and suppose that
$\dual{X} = \aspan^{\trinormdot}(K)$, where $K$ is a Gruenhage
compact in the $\weakstar$-topology and $\trinormdot$ is equivalent
to a coarser, $\weakstar$-lower semicontinuous norm on $\dual{X}$.
Then $\dual{X}$ admits an equivalent, strictly convex dual norm.
\end{prop}

A norm on a dual space $\dual{X}$ is said to be $\weakstar${\em
-LUR} if, given $x$ and $x_n$ in $\dual{X}$ such that $\norm{x} =
\norm{x_n} = 1$ and $\norm{x + x_n} \rightarrow 2$, we have $x_n
\rightarrow x$ in the $\weakstar$ topology. Such norms are studied
in \cite{raja:03}.

\begin{thm}[{\cite[Theorem 1.3]{raja:03}}]\label{rajawstar}If
$\ball{\dual{X}}$ is a descriptive compact space in the
$\weakstar$-topology then $X$ admits an equivalent norm with
$\weakstar$-LUR dual norm.
\end{thm}

The next result is our promised characterisation.

\begin{thm}
\label{fullcharac} Let $X$ have a normalised 1-unconditional basis,
with $\mathscr{A}_n$, $n \in \nat$, defined as above. Then the
following are equivalent.
\begin{enumerate}
\item $(\ball{\dual{X}},\weakstar)$ is a Gruenhage compact space;
\item $X$ admits an equivalent norm with a strictly convex dual norm;
\item there exists a strictly increasing, $\tau_p$-lower semicontinuous
map
$$
\mapping{\rho}{K_\mathscr{A}}{[0,1]};
$$
\item $(\ball{\dual{X}},\weakstar)$ is a descriptive compact space;
\item $X$ admits an equivalent norm with a $\weakstar$-LUR dual norm.
\end{enumerate}
\end{thm}

\begin{proof}
(1) $\Rightarrow$ (2) follows from Theorem \ref{gru}, part (2). (4)
$\Rightarrow$ (5) follows from Theorem \ref{rajawstar}. (5)
$\Rightarrow (2)$ is an immediate consequence of the definition and
(4) $\Rightarrow$ (1) is \cite[Corollary 4]{smith:08}. We prove (2)
$\Leftrightarrow$ (3) and (2) $\Rightarrow$ (4). Suppose that (2)
holds. Using Theorem \ref{sclattice}, we can find a strict lattice
dual norm $\trinormdot$ on $\dual{X}$. It is easy to see that the
map $\ind{A} \mapsto \trinorm{\dind{A}}$ satisfies (4). Now suppose
that $\rho$ satisfies (3). Let $K_n = \setcomp{\dind{A}}{A \in
\mathscr{A}_n}$, where $\mathscr{A}_n$ is as above. By Theorem
\ref{adequate}, each $K_n$ is a descriptive compact in the
$\weakstar$-topology. If we set $K = \bigcup_{n=1}^\infty n^{-2}K_n
\cup \{0\}$ then $K$ is again descriptive, and
$\closure{\aspan}^{\pnormdot{\infty}}(K) =
\closure{W}^{\pnormdot{\infty}} = \dual{X}$, where $W$ is as above.
Since $\pnormdot{\infty}$ is a $\weakstar$-lower semicontinuous norm
on $\dual{X}$, coarser than $\normdot$, we can apply Proposition
\ref{coarserdense} to obtain an equivalent, strictly convex dual
norm on $\dual{X}$.

To finish, we prove (2) $\Rightarrow$ (4). Given (2), let
$\trinormdot$ be as above, and identify the positive part $B_+$ of
its unit ball with a $\tau_p$-compact subset of $[0,1]^\Gamma$. By
applying Lemma \ref{master} to $B_+$ with $\rho = \trinormdot$, we
have that $\setcomp{f \in B_+}{\trinorm{f} = 1}$ has a
$\sigma$-isolated network in the $\tau_p$ (equivalently $\weakstar$)
topology. It follows from \cite[Theorem 7.2]{hansell:01} that
$(B_+,\weakstar)$ has a $\sigma$-isolated network, so is a
descriptive compact. Because descriptive compact spaces are
preserved under continuous images \cite[Corollary 3.4]{or:04}, we
have that $B \subseteq B_+ - B_+$ is descriptive.
\end{proof}

\begin{rem}\label{notmuchl1}
We observe that $A \in \mathscr{A}$ if and only if
$(x_\gamma)_{\gamma \in A} \mapsto \sum_{\gamma \in A} x_\gamma
e_\gamma$ defines an isomorphism from $\lpk{1}{A}$ into $X$. Thus
Theorem \ref{fullcharac}, part (3), is a more precise formulation of
the assertion, made after the proof of Theorem \ref{flurequiv}, that
in the dual strictly convex case $X$ cannot contain too many copies
of $\lp{1}$.
\end{rem}

Finally, we note that there is a Banach space of type
$h_\mathscr{A}$ which does not satisfy the conditions of Theorems
\ref{adequate} or \ref{fullcharac} and does not contain an
isomorphic copy of $\lpk{1}{\Gamma}$ for any uncountable $\Gamma$.
See \cite[Theorems 1.7 and 3.6 (c)]{am:93}.

\section{Examples}\label{examples}

In this section, we apply Lemma \ref{master} and Theorem
\ref{adequate} to obtain some new results concerning examples of
compact spaces from the literature.

\begin{defn}
We shall say that a partially ordered set $(T,<)$ is a {\em
pseudotree} (respectively {\em tree}) if, for every $x \in T$, the
set $I_x = \setcomp{w \in T}{w < x}$ is totally (respectively {\em
well}) ordered.
\end{defn}

Pseudotrees were introduced by Kurepa and studied in \cite{ls:84}
under the name of {\em bushes}. We say that a subset $\Gamma$ of a
partially ordered set is an {\em antichain} if no two distinct
elements of $\Gamma$ are comparable. A partially ordered set is
called {\em special} if it can be written as a countable union of
antichains. Given a pseudotree $T$, we let $\mathscr{A}_T$ be the
adequate family of all totally ordered subsets of $T$. Such families
were investigated in the context of Talagrand compact spaces in
\cite{ls:84}. A compact space $K$ is called {\em Talagrand} if the
Banach space $\Ck{K}$ is $\mathcal{K}${\em -analytic} in its weak
topology; see, for example, \cite{aam:08,ls:84,tal:79} for details.

\begin{prop}[{\cite[Theorem 3.2]{ls:84}}]
Let $T$ be a pseudotree. Then $K_{\mathscr{A}_T}$ is a Talagrand
compact if and only if $T$ is a countable union of antichains.
\end{prop}

We can use Theorem \ref{adequate} to provide a straightforward
extension of this result. Recall that a compact space $K$ is called
{\em Eberlein} if it is homeomorphic to a weakly compact subset of a
Banach space. The implications Eberlein $\Rightarrow$ Talagrand
$\Rightarrow$ descriptive $\Rightarrow$ Gruenhage have been
established and are known to be strict.

\begin{prop}\label{pseudotree}
Let $T$ be a pseudotree. Then the following are equivalent.
\begin{enumerate}
\item $K_{\mathscr{A}_T}$ is Eberlein;
\item $K_{\mathscr{A}_T}$ is Gruenhage;
\item $T$ is special.
\end{enumerate}
\end{prop}

\begin{proof}
Only (2) $\Rightarrow$ (3) and (3) $\Rightarrow$ (1) require proof.
Assume (2). By Theorem \ref{adequate}, there exists a strictly
increasing map $\mapping{\rho}{K_{\mathscr{A}_T}}{[0,1]}$. If we
pick $\sigma(x) \in (\rho(\ind{I_x}),\rho(\ind{I_x \cup \{x\}}))
\cap \rat$ for each $x \in T$, it is evident that
$\mapping{\sigma}{T}{\rat}$ is strictly increasing, and that the
fibres $\sigma^{-1}(q)$, $q \in \rat$, are antichains.

Assume (3). Let $T = \bigcup_{n \in \nat} \Gamma_n$, where
$\Gamma_n$, $n \in \nat$, is a pairwise disjoint family of
antichains. It is clear that the map
$\mapping{\pi}{K_{\mathscr{A}_T}}{\czerok{T}}$, defined by
$\pi(\ind{A})(x) = 2^{-n}$ if $A \cap \Gamma_n = \{x\}$ for some
$n$, and $\pi(\ind{A})(x) = 0$ otherwise, is a homeomorphism of
$\mathscr{A}_T$ and a weakly compact subset of $\czerok{\Gamma}$.
\end{proof}

Incidentally, using Proposition \ref{pseudotree}, we can provide
more examples of spaces $h_\mathscr{A}$ which fail the conditions of
Theorems \ref{adequate} and \ref{fullcharac} and do not contain
isomorphic copies of $\lpk{1}{\Gamma}$ for any uncountable set
$\Gamma$. Let $T$ be a non-special pseudotree with no uncountable
branches. Plenty of such objects exist; for example, the space
$\sigma\rat$ of well ordered subsets of $\rat$, partially ordered by
taking initial segments, satisfies these conditions. By Proposition
\ref{pseudotree}, $h_{K_{\mathscr{A}_T}}$ fails Theorems
\ref{adequate} and \ref{fullcharac}. To see that it does not contain
$\lpk{1}{\Gamma}$ if $\Gamma$ is uncountable, we refer the reader to
\cite[Theorem 1.7 and Proposition 3.10]{am:93} and \cite[Theorem
3.5]{amn:88}.

The following class of examples is based on a compact space
constructed in \cite{am:93}. This construction is shown to be
descriptive in \cite{or:04}. Here, we give an alternative proof
using our strictly increasing functions.

\begin{example}[cf {\cite[Example 4.5]{or:04}}]\label{adqnongulko}
Let $L = \setcomp{(\xi,\eta) \in \omega_1^2}{\xi < \eta}$ and
suppose that we have a function $\mapping{\Phi}{L}{\nat}$. Define
$$
W \;=\; \setcomp{\{\oneton{\xi}{n}\} \subseteq \wone}{\xi_1 < \ldots
< \xi_n \mbox{ and }\Phi(\xi_i,\xi_j) \geq j \mbox{ for all }i < j
\leq n < \omega}
$$
and
$$
\mathscr{A}_\Phi \;=\; \setcomp{A \subseteq \wone}{\mbox{every
finite subset of }A\mbox{ is in }W}.
$$
Then $W$ contains all singleton subsets of $\wone$ and is closed
under taking subsets, and $\mathscr{A}_\Phi$ is adequate. We show
that $K_{\mathscr{A}_\Phi}$ is a descriptive compact space.
\end{example}

\begin{proof}
If $\ind{A} \in K_{\mathscr{A}_\Phi}$ then $A$ cannot have order
type exceeding $\omega$. Indeed, otherwise, then we could extract
elements $\xi_1 < \ldots < \xi_n < \ldots < \xi_\omega \in A$, to
give $\{\xi_1,\ldots,\xi_n,\xi_\omega\} \in W$ and
$\Phi(\xi_1,\xi_\omega) \geq n+1$ for all $n$, which is impossible.

In order to construct a strictly increasing, $\tau_p$-lower
semicontinuous function
$\mapping{\rho}{K_{\mathscr{A}_\Phi}}{\real}$, we first define
$\mapping{\pi}{K_{\mathscr{A}_\Phi}}{\czerok{L}}$ by
$$
\pi(\ind{A})(\xi,\eta) \;=\; \left\{
\begin{array} {l@{\quad}l}
n^{-1} & \mbox{if } \xi,\eta \in  A \mbox{ and }\Phi(\xi,\eta) = n\\
0 & \mbox{otherwise.} \\
\end{array} \right.
$$
It is clear that $\pi(\ind{A}) \in \czerok{L}$ because there are
only finitely many $(\xi,\eta) \in L \cap A^2$ with $\Phi(\xi,\eta)
= n$. Indeed, as we have already observed, we can enumerate $A$ as a
strictly increasing sequence $(\xi_i)_{i < \alpha}$, where $\alpha
\leq \omega$. Thus, if $(\xi,\eta) \in L \cap A^2$ and
$\Phi(\xi,\eta) = n$ then $\xi = \xi_i$ and $\eta = \xi_j$ for some
$i < j$, so $\{\xi_1,\ldots,\xi_j\} \in W$ and $j \leq
\Phi(\xi_i,\xi_j) = n$. Evidently, if $A \subseteq B$ then
$\pi(\ind{A}) \leq \pi(\ind{B})$. It is also clear that if $B$
contains at least two elements and strictly contains $A$, then
$\pi(\ind{A})(\xi,\eta) = 0 < \pi(\ind{B})(\xi,\eta)$ for some
suitable $(\xi,\eta) \in L$. Finally, we observe that $\pi$ is
$\tau_p$-$\tau_p$ continuous.

We define our strictly increasing function $\rho$ by recalling Day's
norm $\normdot_{\mathrm{Day}}$ from the introduction and setting
$$
\rho(\ind{A}) \;=\; \left\{
\begin{array} {l@{\quad}l}
0 & \mbox{if } A = \varnothing\\
1 & \mbox{if } A \mbox{ is a singleton}\\
1 + \norm{\pi(\ind{A})}_{\mathrm{Day}} & \mbox{otherwise.} \\
\end{array} \right.
$$
Since $\normdot_{\mathrm{Day}}$ is $\tau_p$-lower semicontinuous and
lattice, $\rho$ is $\tau_p$-lower semicontinuous. Being strictly
convex, Day's norm is moreover strictly lattice, thus $\rho$ is
strictly increasing. It follows that $K_{\mathscr{A}_\Phi}$ is
descriptive by Theorem \ref{adequate}.
\end{proof}

The compact space $K$ is said to be {\em Gul'ko} if $\Ck{K}$ is
$\mathcal{K}${\em -countably determined} in its weak topology. We
say that $K$ is {\em Corson} if it is homeomorphic to a subset $M$
of $[0,1]^\Gamma$ in the pointwise topology, with the property that
the support of every element of $M$ is at most countable. See, for
example, \cite{ls:84,gru:87,amn:88,am:93,fabian:97} for information
about these classes of compact spaces. The implications Talagrand
$\Rightarrow$ Gul'ko $\Rightarrow$ Corson and descriptive are known
and strict. Since the order type of every $A \in \mathscr{A}_\Phi$
of Example \ref{adqnongulko} cannot exceed $\omega$, the associated
compact space $K_{\mathscr{A}_\Phi}$ is always Corson. Using a
particular function $\Phi$ defined in terms of a family of almost
disjoint subsets of $\omega$, the authors of \cite{am:93} show that
the associated space $K_{\mathscr{A}_\Phi}$ is not Gul'ko.

Our final collection of examples is based on a class of compact
spaces introduced in \cite{aam:08}. We shall say that a subset $I$
of a pseudotree $T$ is an {\em interval} of $T$ if $s \in I$
whenever $r,t \in I$ and $r \prec s \prec t$.

\begin{example}
\label{reznichenko} Let $A$ be a set and $(T_a,<_a)_{a \in A}$ a
family of special pseudotrees with the property that
\begin{enumerate}
\item[($**$)]if $I$ is an interval of $T_a$ and $T_b$ for distinct $a,b \in A$, then $\card{I} \leq 1$.
\end{enumerate}
Let $T = \bigcup_{a \in A} T_a$ and define
$$
\Omega \;=\; \setcomp{\ind{I} \in \{0,1\}^T}{I \mbox{ is an interval
of }T_a \mbox{ for some }a \in A}.
$$
Using ($**$), it is straightforward to prove that $\Omega$ is
$\tau_p$-compact. We show moreover that $\Omega$ is descriptive.
\end{example}

\begin{proof}
Since each $T_a$ is special, we can take a sequence $(A_{a,n})_{n =
1}^\infty$ of pairwise disjoint antichains of $T_a$ such that $T_a =
\bigcup_{n=1}^\infty A_{a,n}$. If $I \in \Omega$ has at least two
elements then there is a unique $a_I \in A$ such that $I \subseteq
\Omega_{a_I}$. Thus we can define
$$
\rho(\ind{I}) \;=\; \left\{
\begin{array} {l@{\quad}l}
0 & \mbox{if } I = \varnothing\\
1 & \mbox{if } I \mbox{ is a singleton}\\
1 + \sum\setcomp{2^{-n}}{I \cap A_{a_I,n} \neq \varnothing} & \mbox{otherwise.} \\
\end{array} \right.
$$
It is evident that $\rho$ is $\tau_p$-lower semicontinuous. It
remains to prove that $\rho$ satisfies property (\ref{rhocondition})
of Lemma \ref{master}. Suppose that $\ind{J} < \ind{I}$. If $I$ is
the singleton $\{t\}$ then set $\alpha = \frac{1}{2}$. Otherwise $I
\subseteq \Omega_{a_I}$. In this case take $t \in I\backslash J$ and
set $\alpha = \rho(\ind{I}) - 2^{-(m+1)}$, where $m$ is the unique
natural number satisfying $t \in A_{a_I,m}$. Note that $\alpha > 1$.
In both cases, define $U = \setcomp{\ind{R} \in \Omega}{t \notin
R}$. It is straightforward to verify that $\alpha$ and $U$ fulfil
property (\ref{rhocondition}) of Lemma \ref{master}.
\end{proof}

The authors of \cite{aam:08} use families of trees $(T_a,<_a)_{a \in
A}$ satisfying ($**$) of Example \ref{reznichenko} to produce
several examples of compact spaces $\Omega$ in this way, including a
non-Gul'ko space. Every tree considered in \cite{aam:08} has height
$\omega$, so is certainly special. If each $(T_a,<_a)$ is special
then none contain uncountable intervals, so the corresponding
compact space $\Omega$ is Corson. From above, it follows that
$\Omega$ is also descriptive. In particular, we have another example
of a Corson, descriptive, non-Gul'ko space.

\bibliographystyle{amsplain}

\end{document}